\documentclass[10pt]{amsart}
\usepackage[utf8]{inputenc}

\newcommand\CC{\mathbb C}

\newcommand\NN{\mathbb N}
\newcommand\DD{\mathcal D}

\newtheorem*{thm*}{Theorem}
\newtheorem{thm}{Theorem}[section]
\newtheorem{lem}[thm]{Lemma}
\newtheorem{prop}[thm]{Proposition}
\newtheorem{cor}[thm]{Corollary}

\theoremstyle{definition}
\newtheorem{defn}[thm]{Definition}

\theoremstyle{remark}
\newtheorem{remark}[thm]{Remark}

\theoremstyle{plain}

\title[The Algebra of Differential Operators for a Matrix Weight.\\ An ultraspherical example.]{The Algebra of Differential Operators for a Matrix Weight.\\ An ultraspherical example.}

\author{Ignacio N. Zurri\'an}

\address{Facultad de Matem\'aticas,
Pontificia Universidad Cat\'olica de Chile. Macul, Santiago, CHILE.}  

\email{zurrian@famaf.unc.edu.ar}  

\thanks{ This research was supported through the programme  ``Oberwolfach Leibniz Fellows" by the Mathematisches Forschungsinstitut Oberwolfach in 2015 and CONICET Argentina. }



\begin{document}

\begin{abstract}
In this paper we study in detail  algebraic 
properties of the algebra $\mathcal D(W)$ of differential operators  associated to a matrix weight of Gegenbauer  type. 
We prove that two second order operators generate the algebra, indeed $\mathcal D(W)$ is isomorphic to the free algebra generated by two elements subject to certain relations. Also, the center is isomorphic to the affine algebra of a singular rational curve. The algebra $\mathcal D(W)$ is a finitely-generated torsion-free module over its center, but it is not flat and therefore it is not projective.  

This is the second  detailed study of an algebra $\mathcal D(W)$ and the first one coming from spherical functions and group representations. We prove that the algebras  for different Gegenbauer weights and the algebras studied previously, related to Hermite weights, are isomorphic to each other.
We give some general results that allow us
to regard the algebra $\mathcal D(W)$ as the centralizer of its center in the Weyl algebra. 
We do believe that this should hold  for any  irreducible weight and the case considered in this paper represents a good step in this direction.
\end{abstract}

\maketitle

\section{Introduction}

The subject of matrix-valued orthogonal polynomials was introduced by M.G. Krein more than sixty years ago, see \cite{K49,K71}. In the scalar valued context, the orthogonal polynomial satisfying second order differential equations played a very important role in many areas of mathematics and its applications. 
A.J. Dur\'an started the study of matrix-weights whose inner product admits a symmetric second order matrix differential operator in \cite{D97}, following similar considerations by S. Bochner in the scalar case in \cite{B29}.  A general approach to the case of higher order can be found in \cite{DG86,GH97}.



Motivated by J.A. Tirao's work on the theory of matrix-valued spherical functions,
see \cite{T77}, one finds in \cite{GPT02a,GPT03,G03} the first examples of situations as those considered in \cite{D97};  in \cite{DG04} one can find the development of a general method that leads to introduce and study classes of examples of orthonormal matrix polynomials.
In the last years, one can see a growing number of papers devoted to
different aspects of this question. For some of these recent papers, see 
\cite{GPT01,GPT02b,T03,GPT04,DG05a,DG05b,DG05c,DG05d,
CG05,GPT05,M05,PT07a}
as well as 
\cite{DI08,DdI08,PR08,GIV11,I11,KvPR12,KvPR13,CI14,vPR14,AKR15}.


For a given matrix weight $W(x)$ of size $N$ on the real line, on can consider a sequence of matrix orthogonal polynomials with respect to $W$. Then, it is natural to think in  the  algebra $\mathcal D(W) $ of matrix differential operators having as eigenfunctions all these polynomials.
%
A
first attempt to go beyond the existence of nontrivial elements in $\mathcal D(W)$ and to study the full algebra (experimentally, with the assistance of symbolic computation)
is undertaken in \cite{CG06}.
In \cite{T11} we can find a deep study of an algebra $\mathcal D(W)$. In that case the author considered the simplest possible example, which was one of the five examples included earlier in \cite{CG06}.

We wish to understand some algebraic properties of the algebra $\mathcal D(W)$ and in this work we consider an example coming from group representation theory.
 More precisely, our weight $W$ and its orthogonal polynomials were studied in \cite{PZ15} based on \cite{TZ14b}. Namely,
$$
  W(x)= W_{p,n}=(1-x^2)^{\tfrac n2 -1} \begin{pmatrix}
  p\,x^2+n-p & -nx\\ -nx & (n-p)x^2+p
\end{pmatrix},\quad x\in [-1,1],
$$ for real parameters $0<p<n/2$.
This was the example employed in the study of the first case of the time and band limiting problem for matrix-valued functions (see \cite{GPZ15}).
    
This algebra is of interest by itself. 
Even though it is not commutative and in the context of bispectrality one of the variables is discrete, there are strong motivations coming from the classical theory of commuting operators.  
Among all the equations that are considered in the framework of the inverse scattering method we can mention the Korteweg-de Vries equation and its two dimensional generalization, the Kadomtsev-Petviashvili equation, the non-linear Schr\"odinger equation, the sine-Gordon equation and many other fundamental equations of modern mathematical physics. 
All these equations can be represented as compatibility conditions for an overdetermined system of auxiliary linear problems. For example, for the KdV equation, this system has the form 
$$
D\psi=0,\qquad  \psi_t=E\psi,
$$
with 
$$
D=-\tfrac{\partial^2}{\partial x^2}+u(x,t),\qquad E=\tfrac{\partial^3}{\partial x^3}-\tfrac32\,u\,\tfrac{\partial}{\partial x}-\tfrac34\, u_x.
$$
The compatibility of this system implies 
\begin{equation*}
\left[\tfrac{\partial}{\partial t}-E ,D\right]=0\qquad  \Longleftrightarrow\qquad D_t=\left[E,D\right].
\end{equation*}
This operator equation is called ``Lax equation''; each Lax equation is an infinite-dimensional analogue of the completely integrable systems. For the KdV equation, Novikov considered the restriction to the space of stationary solutions, which is essentially equivalent to consider a particular case of the more general problem of the classification of commuting ordinary differential operators with scalar coefficients. As a pure algebraic problem it was considered by Burchnall and Chaundy in \cite{BC23,BC28}; 
given two commuting operators $D$ and $E$ of order $s$ and $\ell$, respectively, there exists a polynomial $R(\lambda,\mu)$ in two variables such that 
$$
R(D,E)=0.
$$
If $s$ and $\ell$ are coprime then for each point $q=(\lambda,\mu)$ of the curve $\Gamma$, that is defined in $\CC^2$ by the equation $R(\lambda,\mu)=0$, there corresponds a unique (up to constant factor) common eigenfunction $\psi(x,q)$ of $D$ and $E$:
$$
D\psi(x,q)=\lambda \psi(x,q),
\qquad
E\psi(x,q)=\mu\psi(x,q).
$$
These commuting pairs are classified by a set of algebro-geometric data, see \cite{K81,M77,T05}.

Of course, working in a scalar context is crucial, as one would expect; if we are dealing with a noncommutative ring of coefficients, such as the ring of matrices, things are more complicated. By considering only commuting differential operators of order two one has a very illuminating contrast.
In the scalar case the coefficients of the highest derivatives of the commuting operators have to be the same up to a constant factor, thus one is essentially dealing with an operator of order two and another of order one. In the matrix case this is more complicated. In \cite{KV01} one can find an interesting discussion and results on commuting matrix differential operators. However, a natural source is the algebra $\mathcal D(W)$,
with special attention to its center or even to the abelian maximal subalgebras. Particularly, we prove that our algebra $\mathcal D(W)$ is the centralizer of its center in the Weyl algebra $\mathbf D$ given by 
$${\mathbf D} =\{D=\sum_{i=0}^s\partial^i F_i: F_i \in M_N[x]\}.$$
We prove that for any weight $W$, under certain conditions, this presentation of the algebra in terms of its center is still valid. We expect that
this presentation is always possible for any irreducible weight.



\smallskip

Now, we proceed to describe the contents of each section.
In Section \ref{Pre} we give a  brief review of some general notions and properties of the algebras $\mathcal D(W)$.

In Section \ref{Geg} we introduce our Gegenbauer example.
We prove that our weights $W_{p,n}$, $0<p<n/2$, are non-similar each other.
Even though this weight is very different from the weight considered in \cite{T11}, we take advantage of some ingenious tools used by Tirao to simplify some of the conditions satisfied by the operators in $\mathcal D(W)$ and, at the beginning, we follow the line of argument in \cite{T11}.

In Section \ref{Ord} we prove that for any operator 
$$D=\sum_{\ell=0}^s \partial^\ell F_\ell(x)\in\mathcal D(W),$$
if  $F_s\neq{\bf0}$ then $\operatorname{deg}(F_s)=s$ and $s$ is even. Even more, we have four operators $D_1,D_2,D_3,D_4$ of order two that generate the algebra $\mathcal D(W)$. 

In Section \ref{Str} we prove that $\mathcal D(W)$ is isomorphic to the free algebra generated by two elements, $A$ and $B$, with the following relations
\begin{align*}
B^2A-AB^2&=0,& 
B  {A}^{2}+{A}^{2}  B-2\,A  B  A-B&=0,\\
 B  A  B+{A}^{3}-2\,A  {B}^{2}-A  &=0,&
 {B}^{3}-2{A}^{2}  B+A  B  A  &=0.
\end{align*}
We also prove that all these non-similar weights $W_{p,n}$ have isomorphic algebras $\mathcal D(W_{p,n})$.

Also, in spite of the fact that the operators in \cite{T11} are clearly  different to the operators here, we prove that the algebras are isomorphic. 

In Section \ref{Cen} we focus our attention on the center $\mathcal {Z}(W)$ of the algebra $\mathcal D(W)$, proving that it is generated by two operators ${C_1}$ and ${C_2}$ of order four and six, respectively. Even more, the center of the algebra $\mathcal D(W_{p,n})$ is isomorphic to the affine algebra of the following singular rational curve:
$$
x^3-y^2=(n-2p)\,x\,y.
$$
In the process of determining whether this is an Azumaya algebra or not, we prove that the algebra $\mathcal D(W)$ is a finitely-generated torsion-free module over the ring $\mathcal Z(W)$, but it is not flat and therefore not projective. 

In Section \ref{weyl} we prove that under certain conditions, the algebra $\mathcal D(W)$ is the centralizer of its center in the Weyl algebra.

\smallskip

We are convinced that the study of these algebras can be useful for a better understanding of the theory of matrix orthogonal polynomials that satisfy differential equations of order two. There is recent work \cite{TZ15}, which includes 
 very simple and direct criteria for the reducibility of a matrix weight $W$ based on the algebra $\mathcal D(W)$.
After carrying out  the comparison and deep study of these algebras, we believe that a natural step is to start a classification of matrix differential operators associated with matrix weights, aiming at giving a description of the matrix commuting operators in algebro-geometric terms.
 

\section{Preliminaries}\label{Pre}

Let $W = W (x)$ be a weight matrix of size $N\in\mathbb N$ on the real line. By this we mean
a complex $N \times N$ matrix-valued integrable function on the interval $(a, b)$ such that
$W (x)$ is positive definitive almost everywhere and with finite moments of all orders.
From now on we shall denote by $M_N$ the algebra of all $N\times N$ matrices over $\CC$ and by
$M_N[x]$ the algebra over $\CC$ of all polynomials in $x$ with
coefficients in $M_N$. We denote $\mathbf I$ the identity of $M_N$ and by $T^*$ the transpose  conjugate of $T$ .
We introduce as in \cite{K49} and \cite{K71} the following matrix-valued Hermitian sesquilinear form in the linear space $M_N[x]$:
$$
(P, Q) =\int_a^b P (x)W (x)Q(x)^* dx.
$$

Then $M_N[x]$ is a left inner product $M_N$-module and  there exists
a unique sequence $\{Q_w\}_{w\ge0}$ of monic orthogonal polynomials. More generally, a sequence $\{P_w\}_{w\ge0}$ of matrix orthogonal polynomials is a sequence of
elements $P_w \in M_N[x]$ such that $P_w$ is of degree $w$, its leading coefficient is a non
singular matrix and $(P_w' , P_w) = 0$ for all $w' \ne w$. Then any sequence $\{P_w\}_{w\ge0}$ of
matrix orthogonal polynomials is of the form $P_w = A_w Q_w$ where $A_w \in GL_N (\mathbb C)$ is
arbitrary for each $w \ge 0$.
We come now to the notion of a differential operator with matrix coefficients
acting on matrix-valued polynomials, i.e. elements of $M_N[x]$. These operators can
be made to act on our functions either on the left or on the right (see \cite{D97}), but it turns out to be convenient  to work with differential operators acting on the right. Then, 
$D$ given by

$$D=\sum_{i=0}^s\partial^i F_i(x),\quad  \partial =\frac{d}{dx} ,
$$
acts on $P (x)$ by means of

$$PD(x)=\sum_{i=0}^s\partial^i(P)(x) F_i(x),\quad  \partial =\frac{d}{dx} .
$$
%
Notice that if $D_1$ and $D_2$ are two differential operators we have
$P (D_1D_2 ) = (P D_1)D_2$.
The following three propositions are taken from \cite{GT07}.

\begin{prop}[\cite{GT07}, Proposition 2.6]\label{Ddet}
 Let $W = W (x)$ be a weight matrix of size $N$ and let $\{Q_w\}_{w\ge0}$
be the sequence of monic orthogonal polynomials in $M_N[x]$. If
$$D=\sum_{i=0}^s\partial^i F_i(x),\quad  \partial =\frac{d}{dx} ,
$$
is a linear right-hand side ordinary differential operator of order s such that
\begin{equation}\label{auto}
 Q_wD = \Lambda_wQ_w, \qquad
\text{ for all }
\quad
 w\ge 0,
\end{equation}
with $\Lambda_w \in M_N$, then $F_i = F_i (x) \in M_N[x]$ and $\deg F_i \le i$. Moreover $D$ is determined by
the sequence $\{\Lambda_w\}_{w\ge0}.$
\end{prop}
We could have written the eigenvalue matrix $\Lambda_w$ to the right of the matrix
valued polynomials $Q_w$ above. However, as shown in \cite{D97}, this only leads to uninteresting cases where the weight matrix is diagonal. 
To ease the notation, if $\nu\in \mathbb{C}$ let
$$[\nu]_i = \nu(\nu- 1)\cdots(\nu - i + 1),\qquad [\nu]_0 = 1.$$

\begin{prop}[\cite{GT07}, Proposition 2.7]\label{22}
 Let $D=\sum_{i=0}^s\partial^i F_i(x)$ satisfy \eqref{auto}, with 
 $$F_i(x)=\sum_{j=0}^iF_j^i(x).$$
 Then 
 \begin{equation}
  \Lambda_w=\sum_{j=0}^s [w]_i F_i^i\qquad \text{for all }\quad w\ge0.
 \end{equation}
Hence, $w\to \Lambda_w$ is a matrix-valued polynomial function
with $\deg \Lambda_w\le \deg (D).$
 \end{prop}

 Given a sequence of matrix orthogonal polynomials $\{Q_w\}_{w\ge0}$ we are interested
in the algebra $\DD(W)$ of all right-hand side differential operators with matrix-valued coefficients that have the polynomials $Q_w$ as eigenfunctions. Notice that if $Q_wD = \Lambda_wQ_w$ for some eigenvalue matrix $\Lambda_w \in M_N$, then $\Lambda_w$ is uniquely determined by $D$. In such a case we write $\Lambda_w(D) = \Lambda_w$. Thus

\begin{equation}\label{D(W)}
 \DD(W) = \{D : Q_wD = \Lambda_w(D)Q_w,\, \Lambda_w(D) \in M_N \text{ for all } w \ge 0\}.
\end{equation}

First of all we observe that the definition of $\DD(W)$ depends only on the weight
matrix $W = W(x)$ and not on the sequence $\{Q_w\}_{w\ge0}$. 
 
 \begin{prop}[\cite{GT07}, Proposition 2.8]
  Given a sequence $\{Q_w\}_{w\ge0}$ of orthogonal polynomials let the algebra $\DD(W)$ be as in \eqref{D(W)}. Then $D\to \Lambda_w(D)$ is a representation of
$\DD(W)$ into $M_N$, for each $w \ge 0$. Moreover the sequence of representations $\{\Lambda_w\}_{w\ge0}$
separates the elements of $\DD(W)$.
 \end{prop}

In particular, if $\{Q_w\}_{w\ge0}$ is the sequence of monic orthogonal polynomials we
have a homomorphism
$$\Delta = \prod\Delta_w : \DD(W) \to {M_N}^{\NN_0}$$
of $\DD(W)$ into the direct product of $\NN_0$ copies of $M_N$. Moreover $\Delta$ is injective.

In Section 3 of \cite{GT07} the ad-conditions coming from the bispectral pairs $(L,D)$ are used
to described the image $\Delta(W)$ of $\DD(W)$ by the eigenvalue isomorphism 
$\Delta$.
Here $L$ is the difference operator associated to the three term recursion relation
satisfied by the sequence of monic orthogonal polynomials and $D \in \DD(W)$.
This gives
a completely different presentation of $\DD(W)$ and will be used frequently to simplify
several computations.
Let us mention that the ad-conditions were first introduced in \cite{DG86}.

\section{The $2\times2$ Gegenbauer Example}\label{Geg}

From now on we will consider the matrix-valued orthogonal polynomials related to the $2\times2$ weight matrix given by
\begin{equation}\label{peso}
  W(x)= W_{p,n}=(1-x^2)^{\tfrac n2 -1} \begin{pmatrix}
  p\,x^2+n-p & -nx\\ -nx & (n-p)x^2+p
\end{pmatrix},\quad x\in [-1,1],
\end{equation}
for real parameters $0<p<n/2$.

The orthogonal polynomials arising here are discussed in full in \cite{PZ15}.
As we said in the introduction, the purpose of this paper is to compute in this
case the algebra $\mathcal D(W )$ and to study its structure.
The monic orthogonal polynomials $\{Q_w\}_{w\ge0}$ with respect to the weight matrix
$W (x)$ are explicitly given by
\begin{equation}\label{monicos}
Q_w=   \frac {w!}{2^{w} \left(\tfrac{n+1}2\right)_{w}}  
\begin{pmatrix}
 C_w^{\frac{n+1}{2}}(x)+\frac{n+1}{p+w}\,C_{w-2}^{\frac{n+3}{2}}(x)&\frac{n+1}{p+w}\,C_{w-1}^{\frac{n+3}{2}}(x)\\ \mbox{} \\
\frac{n+1}{n-p+w}\,C_{w-1}^{\frac{n+3}{2}}(x)& C_w^{\frac{n+1}{2}}(x)+\frac{n+1}{n-p+w}\,C_{w-2}^{\frac{n+3}{2}}(x)
\end{pmatrix},
\end{equation}
where 
$(a)_w=a(a+1)\dots (a+w-1)$ denotes the Pochhammer's symbol and
   $C_n^\lambda(x)$ denotes the $n$-th Gegenbauer polynomial
$$C_n^\lambda(x)=\frac{(2\lambda)_n} {n!}\,
 {}_2\!F_1\left(\begin{matrix}-n,\,n+2\lambda\\ \lambda+1/2\end{matrix};\frac{1-x}{2}\right).
$$
As usual, we assume $C_n^\lambda(x)=0$ if $n<0$; recall that $C_n^\lambda$ is a polynomial of degree $n$, with leading coefficient  $\frac{2^n(\lambda)_n}{n!}$.
Recall that (see \cite[p 561]{AS65})
\begin{equation}\label{ec}
C_n^{(\alpha)}(z)=\sum_{k=0}^{\lfloor n/2\rfloor} (-1)^k\frac{\Gamma(n-k+\alpha)}{\Gamma(\alpha)k!(n-2k)!}(2z)^{n-2k}.
\end{equation}

Let us recall the concept of similarity for matrix weights.
Two weights $W$ and $\widetilde W$ are said to be {\em similar} if there exists a nonsingular matrix $M$, which does not depend on $x$, such that
$$\widetilde W(x)=M W(x)M^*, \quad \text{ for all } x\in (a,b).$$

\begin{lem}\label{neq}
 If $(p,n)\neq( \tilde p , \tilde n)$ then $W_{p,n}$ is not similar to $W_{\tilde p,\tilde n}$.
\end{lem}
\begin{proof}
If $W_{p,n}$ is  similar to $W_{\tilde p,\tilde n}$ then there exists a nonsingular matrix
$M=\left(\begin{smallmatrix}
a&b\\c&d
\end{smallmatrix}\right)$
such that $\widetilde W=M WM^*$,
therefore it is clear that $n=\tilde n$. Hence we have that the matrix 
$\left(\begin{smallmatrix}
 \tilde p\,x^2+n-\tilde p & -nx\\ -nx & (n-\tilde p)x^2+\tilde p
\end{smallmatrix}\right)$
 is of the form
$$\left( \begin {smallmatrix}  \left(   \left| a \right| 
^{2}p+  \left| b \right|^{2}n-  \left| b
 \right|^{2}p \right) {x}^{2}+ \left( -\overline{a}bn-
\overline{b}an \right) x+  \left| a \right|^{2}n-
  \left| a \right|^{2}p+  \left| b \right| ^{2}p& \left( \overline{c}ap+ \left( bn-bp \right) 
\overline{d} \right) {x}^{2}+ \left( -bn\overline{c}-an\overline{d}
 \right) x+ \left( an-ap \right) \overline{c}+\overline{d}bp
\\ \noalign{\medskip} \left( \overline{a}cp+ \left( dn-dp \right) 
\overline{b} \right) {x}^{2}+ \left( -dn\overline{a}-cn\overline{b}
 \right) x+ \left( cn-cp \right) \overline{a}+\overline{b}dp& \left( 
  \left| c \right|^{2}p+  \left| d \right|^{2}n-  \left| d \right|^{2}p \right) {x}^{2
}+ \left( -\overline{c}dn-\overline{d}cn \right) x+  \left| c
 \right|^{2}n-  \left| c \right| ^{2}p+
  \left| d \right|^{2}p\end {smallmatrix} \right). 
$$
Then, by looking at the entry $(1,2)$ in the matrix above, we have 
\begin{align*}\left( \overline{c}ap+ \left( bn-bp \right) 
\overline{d} \right)=0& &\text{and}& & \left( an-ap \right) \overline{c}+\overline{d}bp=0,
\end{align*}
but since $M$ is invertible and $0<p<n/2$ we have that $M$ has to be the identity matrix. Hence $(p,n)=( \tilde p , \tilde n)$.
\end{proof}


%
%

\bigskip

Let us denote $$Q_w=\sum_{i=0}^w x^i B_i^w.$$
From \eqref{monicos} and \eqref{ec} we obtain

\begin{equation}\label{par}
 B^w_{w-2k}=
 \frac{w!(-1)^k 2^{-2k} }{(\frac{n+1}{2}+w-k)_k k! (w-2k)! }
\left(\begin{matrix}
\frac{p+w-2k}{p+w} &0\\
0& \frac{n-p+w-2k}{n-p+w}
\end{matrix}\right),
\end{equation}
\begin{equation}\label{impar}
 B^w_{w-2k-1}=
 \frac{w!(-1)^k 2^{-2k} }{(\frac{n+1}{2}+w-k)_k k! (w-2k-1)!}
 \left(\begin{matrix}
0& \frac{1 }{ p+w} \\
 \frac{1 }{n-p+w}&0
\end{matrix}\right).
\end{equation}

We will use some tools already employed by Tirao to simplify some of the conditions 
satisfied by the operators in $\mathcal D(W)$. We quote a result in \cite{T11} (see Proposition 3.4 and equation (25)), making some changes in the parameters.
 \begin{prop}\label{simpl}
 Let $D=\sum_{j=0}^{s}\partial F_j(x)$, with $F_j=\sum_{i=0}^{j}x^i F_i^j$. Then $D\in\mathcal D(W)$ if and only if 
\begin{equation*}
 \sum_{i=0}^{s} B_{w-m+i}^w \left(\sum_{\ell=0}^{s} [w-m+i]_{\ell} F_{\ell-i}^\ell \right)-\left(\sum_{\ell=0}^{s}[w]_\ell F_\ell^\ell \right)B_{w-m}^w=0
\end{equation*}
for all $0\le m\le w$, $0\le w$.
\end{prop}

The following lemma states a nice property of the weight $W$ that can be easily proved.
\begin{lem} Given 
$$T=\begin{pmatrix}
    1&0\\0&-1
   \end{pmatrix}
$$
we have that 
$$ W(-x)=TW(x)T.$$
\end{lem}
\begin{defn}
 Given $D=\sum_{0\le i\le s}\partial^i F_i \in \mathcal D$ let 
$$
\widetilde D=\sum_{0\le i\le s}\partial^i (-1)^i T \widetilde{F_i} T \in \mathcal D,
$$
where $\widetilde {F_i}(x)={F_i}(-x)$.
\end{defn}

%

\begin{prop}[\cite{T11}, Proposition 3.9]\label{L}
If $D\in\mathcal D(W)$ then $\widetilde D\in\mathcal D(W)$. Moreover, if $\{Q_w\}_{w\ge0}$ is the sequence of monic orthogonal polynomials with respect to $W$ and $\{\Lambda_w(D)\}_{w\ge0}$ is the corresponding eigenvalue sequence, then
$$
\Lambda_w\left(\widetilde D\right)=T\Lambda_w(D) T.
$$
\end{prop}

\begin{thm}[\cite{T11}, Theorem 3.10]
  The map $D\to \widetilde D$ defines an involutive automorphism of the algebra 
  $\mathcal D(W)$.
\end{thm}

We will need the following lemma from \cite{PZ15}.
\begin{lem}[\cite{PZ15}, Corollary 5.3]\label{PZ}
The set of differential operators of order at most two in $\mathcal D(W)$ is a five-dimensional vector space generated by the identity $\mathbf{I}$ and 
\begin{align*}
  D_1&= \partial ^2 \begin{pmatrix}  x^2& x\\-x&-1 \end{pmatrix}
     + \partial \begin{pmatrix} (n+2)x&n-p+2 \\-p &0\end{pmatrix}
     +\begin{pmatrix} p\,(n-p+1)&0\\0& 0  \end{pmatrix},\\ \mbox{} \displaybreak[0]\\
  D_2&=\partial^2\begin{pmatrix}  -1& -x\\x&x^2 \end{pmatrix}
     + \partial \begin{pmatrix} 0&p-n\\p+2 &(n+2)x\end{pmatrix}
    +\begin{pmatrix}  0&0\\0& (p+1)(n-p) \end{pmatrix}  ,\\ \mbox{} \displaybreak[0]\\
  D_3&=\partial^2 \begin{pmatrix}  -x& -1\\ x^2&x \end{pmatrix}
     +\partial \begin{pmatrix} -p& 0\\2(p+1)x &p+2\end{pmatrix}
    +\begin{pmatrix}  0&0\\p(p+1) & 0 \end{pmatrix} ,\\ \mbox{} \displaybreak[0]\\
  D_4&=\partial^2\begin{pmatrix}  x&x^2 \\-1&-x \end{pmatrix}
     + \partial \begin{pmatrix} n-p+2 & 2(n-p+1) x\\0 &p-n\end{pmatrix}
    +\begin{pmatrix}  0&  (n-p)(n-p+1)\\0& 0 \end{pmatrix}.
\end{align*}
The corresponding eigenvalues are given by
\begin{equation}\label{eigen}
\begin{aligned}
  \Lambda_w(D_1)&=
  \left(\begin{smallmatrix}
  (w+p)(w+n-p+1) & 0\\ 0& 0
\end{smallmatrix}\right), \, 
&&\Lambda_w(D_2) = \left(\begin{smallmatrix}
 0 & 0\\ 0& (w+p+1)(w+n-p)
\end{smallmatrix}\right) ,\\
\Lambda_w(D_3)& = \left(\begin{smallmatrix}
 0 & 0\\ (w+p)(w+p+1) & 0
\end{smallmatrix}\right), \, 
&&\Lambda_w(D_4) = \left(\begin{smallmatrix}
 0 & (w+n-p)(w+n-p+1)\\ 0& 0
\end{smallmatrix}\right).
\end{aligned}
\end{equation}
\end{lem}

It can be easily checked that 
\begin{align*}
\widetilde D_1 =D_1,
&&
\widetilde D_2 =D_2,
&&
\widetilde D_3 =-D_3
&&
\widetilde D_4 =-D_4,
\end{align*}
It is worthwhile  paying attention to the expression of the matrix eigenvalues in Lemma \ref{PZ}. They will be used in some proofs and they imply the following additional remark.
\begin{remark}\label{rem} 
We have the following relations among the differential operators $D_1, D_2,D_3,D_4$.
\begin{align*}
D_1D_2 =0,&&      D_1D_3&=0,  && D_1D_4= D_4D_2-(n-2p)D_4,  \\
D_2D_1=0,&&  D_2D_3&=D_3D_1+(n-2p)D_3,&& D_2D_4=0,\\
D_3D_2=0,&&  D_3D_3&=0, && D_3D_4= D_2^2-(n-2p)D_2,\\
D_4D_1=0,&&  D_4D_3&=D_1^2+(n-2p)D_1, && D_4D_4=0.
\end{align*}
\end{remark}
\section{The Order of the Operators in $\mathcal D(W)$}\label{Ord}

In this section we prove that any operator 
$$D=\sum_{\ell=0}^s \partial^\ell F_\ell(x)\in\mathcal D(W)$$
is of even order and either $F_s={\bf0}$ or $\operatorname{deg}(F_s)=s$. Furthermore, $\{D_1,D_2,D_3,D_4\}$ generates the algebra $\mathcal D(W)$.
We start by considering operators $D$ such that $\widetilde D=D$, later we prove that the same results hold for all operators.

From \eqref{par} and \eqref{impar} we have
\begin{align}
\left(B_{w-2(k-i)}^w\right)_{11}&=
\frac{p+w-2k+2i}{p+w-2k} 
\frac{(-1)^i 2^{2i}\left(\frac{n+1}{2}+w-k\right)_i [k]_i}{ [w-2k+2i]_{2i}}
\left(B_{w-2k}^w\right)_{11}\label{par2},
\\
\left(B_{w-2k-1}^w\right)_{12}&=
\frac{w-2k}{p+w-2k}
\left(B_{w-2k}^w\right)_{11}\label{impar2}.
\end{align}

 Let us assume that $\widetilde D=D$ and that $s=2r+1$, $r\in\mathbb N$. If, for $0\le \ell \le s$, we denote $F_\ell=\sum_{i=0}^\ell x^i F_i^\ell$ then from our hypotheses we have
  $$
 F_{\ell-2i}^\ell=
 \begin{pmatrix}
    u_{\ell-2i}^\ell&0
    \\
    0&v_{\ell-2i}^\ell
   \end{pmatrix}  ,
   \qquad
 F_{\ell-2i+1}^\ell=
 \begin{pmatrix}
    0&y_{\ell-2i+1}^\ell
    \\
    z_{\ell-2i+1}^\ell&0
   \end{pmatrix}  .
 $$
 
 If we put $m=2k$, Proposition \ref{simpl} gives us that 
 \begin{multline*}
 \sum_{i=0}^{r} B_{w-2k+2i}^w \left(\sum_{\ell=0}^{s} [w-2k+2i]_{\ell} F_{\ell-2i}^\ell \right)+
 \sum_{i=1}^{r+1} B_{w-2k+2i-1}^w \left(\sum_{\ell=0}^{s} [w-2k+2i-1]_{\ell} F_{\ell-2i-1}^\ell \right)
 \\
 -\left(\sum_{\ell=0}^{s}[w]_\ell F_\ell^\ell \right)B_{w-2k}^w=0
\end{multline*}
for $0\le 2k\le w$.
 
 Then, looking at the entry $(1,1)$  and using \eqref{par2} and \eqref{impar2} we have 
 \begin{align*}
  \sum_{i=0}^r 
\frac{p+w-2k+2i}{p+w-2k} 
\frac{(-1)^i 2^{2i}\left(\frac{n+1}{2}+w-k\right)_i [k]_i}{ [w-2k+2i]_{2i}}
\left(B_{w-2k}^w\right)_{11}\;
\left(
\sum_{\ell=2i}^s [w-2k+2i]_{\ell}\; u_{\ell-2i}^\ell
\right)&
\\
+\quad
  \sum_{i=1}^{r+1} 
\frac{w-2k+2i}{p+w-2k} 
\frac{(-1)^i 2^{2i}\left(\frac{n+1}{2}+w-k\right)_i [k]_i}{ [w-2k+2i]_{2i}}
\left(B_{w-2k}^w\right)_{11}\;
\left(
\sum_{\ell=2i-1}^s [w-2k+2i-1]_{\ell}\; z_{\ell-2i+1}^\ell
\right)&
\\
-
\sum_{\ell=0}^s [w]_\ell\; u_\ell^\ell 
\left(B_{w-2k}^w\right)_{11}&
=0.
 \end{align*}
for  $0\le 2k\le w$.
Hence
 \begin{equation*}
  \begin{aligned}
  \sum_{i=0}^r 
{(p+w-2k+2i)(-1)^i 2^{2i}\left(\tfrac{n+1}{2}+w-k\right)_i [k]_i}
\left(
\sum_{\ell=2i}^s [w-2k]_{\ell-2i}\; u_{\ell-2i}^\ell
\right)
&
\\
+
  \sum_{i=1}^{r+1} 
{(-1)^i 2^{2i}\left(\tfrac{n+1}{2}+w-k\right)_i [k]_i}
\left(
\sum_{\ell=2i-1}^s [w-2k]_{\ell-2i+1}\; z_{\ell-2i+1}^\ell
\right)&
-
\sum_{\ell=0}^s [w]_\ell\; u_\ell^\ell 
(p+w-2k)
=0.
 \end{aligned}
 \end{equation*}
 for  $0\le 2k\le w$.
 
Let us consider $w=\widetilde m+2k $, therefore for any $\widetilde m$ and $k$ in $\mathbb N_0$ we have
 \begin{equation}\label{pol}
  \begin{aligned}
  \sum_{i=0}^r 
{(-1)^i 2^{2i}\left(\tfrac{n+1}{2}+{\widetilde m}+k\right)_i [k]_i}
(p+{\widetilde m}+2i)
\sum_{\ell=2i}^s [{\widetilde m}]_{\ell-2i}\; u_{\ell-2i}^\ell&
\\
+
  \sum_{i=1}^{r+1} 
{(-1)^i 2^{2i}\left(\tfrac{n+1}{2}+{\widetilde m}+k\right)_i [k]_i}
\sum_{\ell=2i-1}^s [{\widetilde m}]_{\ell-2i+1}\; z_{\ell-2i+1}^\ell
&
-
(p+{\widetilde m})
\sum_{\ell=0}^s [{\widetilde m}+2k]_\ell\; u_\ell^\ell 
=0.
 \end{aligned}
 \end{equation}
If we consider \eqref{pol} as a polynomial in $k$ then we obtain  \begin{equation}\label{uz0}
z_0^s=0=u_s^s,
\end{equation}
by looking at the terms of degree $s+1$ and $s$.

We observe that \eqref{pol} is a polynomial equation in two variables $\widetilde m$ and $k$ and its term of highest total degree gives us that the following equation holds.

   \begin{align}\label{rep}
     \sum_{i=1}^{r} 
(-1)^i 2^{2i}({\widetilde m}+k)^i k^i
{\widetilde m}^{s-2i+1}\;
\left( u_{s-2i}^s + z_{s-2i+1}^s\right)
=0,\end{align}
thus
\begin{align}
   \nonumber
  \sum_{i=1}^{r} 
    \sum_{j=0}^{i}
     \binom {i} {j}
(-1)^i 2^{2i}({\widetilde m})^{i-j}k^j k^i
{\widetilde m}^{s-2i+1}\;
\left( u_{s-2i}^s + z_{s-2i+1}^s\right)
=0,\end{align}
hence
\begin{align}\nonumber
  \sum_{i=1}^{r} 
    \sum_{j=0}^{i}
     \binom {i} {j}
(-1)^i 2^{2i}k^{j+i}
{\widetilde m}^{s-i-j+1}\;
\left( u_{s-2i}^s + z_{s-2i+1}^s\right)
=0.
 \end{align}
 
 Then, for any  $1\le q\le r$ we observe that the term of degree $q$ in the variable $k$ is 
 
 \begin{equation*}
  \sum_{\frac q2\le i\le q} 
         \binom {i} {q-i}
(-1)^i 2^{2i}k^{q}
{\widetilde m}^{s-q+1}\;
\left( u_{s-2i}^s + z_{s-2i+1}^s\right)
=0.
 \end{equation*}
Hence, for any $1\le q\le r$ we have
\begin{equation}\label{pol2}
  \begin{aligned}
    \sum_{\frac q2\le i\le q}
     \binom {i} {q-i}
(-1)^i 2^{2i}
\left( u_{s-2i}^s + z_{s-2i+1}^s\right)
=0.\\
 \end{aligned}
 \end{equation}
By letting $q=1$ in \eqref{pol2} we have $u_{s-2}^s + z_{s-1}^s=0$. Then, by letting $q=2$ also in \eqref{pol2}  we have $u_{s-4}^s + z_{s-3}^s=0$. Inductively we have 
\begin{equation}\label{uz}
 u_{s-2i}^s =- z_{s-2i+1}^s, \text{ for  } 1\le i \le r.
\end{equation}
In a completely analogous way one finds that
\begin{equation}\label{vy0}
 y_0^s=0=v_s^s
\end{equation}
 and that
\begin{equation}\label{vy}
 v_{s-2i}^s =- y_{s-2i+1}^s, \text{ for  } 1\le i \le r.
\end{equation}
 On the other hand, from Proposition \ref{L}, we know that $\Lambda_w(D)$ is a diagonal matrix for any $w$, thus $D$ commutes with $D_1$. By looking at the highest order derivative in the equation $DD_1=D_1D$  we obtain
 $$
 \begin{pmatrix}  x^2& x\\-x&-1 \end{pmatrix}
 F_s
 =
 F_s
 \begin{pmatrix}  x^2& x\\-x&-1 \end{pmatrix}
 ,
 $$
i.e.
\begin{align*}
 \begin{pmatrix}  x^2& x\\-x&-1 \end{pmatrix}
\sum_{k=0}^{s}x^k F_k^s
 =&
 \sum_{k=0}^{s}x^k F_k^s
 \begin{pmatrix}  x^2& x\\-x&-1 \end{pmatrix},
 \end{align*}
then, for any $0\le k \le s$, we have 
\begin{equation}\label{k}
 \left(\begin{smallmatrix}  0&0\\0&-1\end{smallmatrix}\right)
 F_k^s
 + 
 \left(\begin{smallmatrix}  0&1\\-1&0 \end{smallmatrix}\right)
 F_{k-1}^s
 +
 \left(\begin{smallmatrix}  1&0\\0&0 \end{smallmatrix}\right)
 F_{k-2}^s
 =
 F_k^s
 \left(\begin{smallmatrix}  0&0\\0&-1\end{smallmatrix}\right)
 + 
 F_{k-1}^s
 \left(\begin{smallmatrix}  0&1\\-1&0 \end{smallmatrix}\right)
 +
 F_{k-2}^s
\left(\begin{smallmatrix}  1&0\\0&0 \end{smallmatrix}\right).
 \end{equation}
If $k$ and $s$ are of the same parity, say $k=s-2i+2$, then  \eqref{k} becomes
$$
 \left(\begin{smallmatrix}  0&1\\-1&0 \end{smallmatrix}\right)
 F_{s-2i+1}^s
 =
 F_{s-2i+1}^s
 \left(\begin{smallmatrix}  0&1\\-1&0 \end{smallmatrix}\right)
  $$
whence, looking at the entry $(1,1)$, we obtain 
\begin{equation}\label{zy}
 z_{s-2i+1}^s =- y_{s-2i+1}^s, \text{ for  } 1\le i \le r+1.
\end{equation}

If $k$ and $s$ are of different parity, say $k=s-2i+1$, then the entry $(1,2)$ of the matrix equation \eqref{k} gives us 
\begin{equation*}
 -z_{s-2i+1}+u_{s-2i}
=
-v_{s-2i}+z_{s-2i-1},
\end{equation*}
which combined with \eqref{uz}, \eqref{vy} and \eqref{zy} becomes
\begin{equation}\label{+1-1}
u^s_{s-2i}
=
-u_{s-2i-2}^s.
\end{equation}
 
Furthermore, using \eqref{uz}, \eqref{vy} and \eqref{zy} it is
clear that the values $u_1^s, u_3^s, \dots , u_s^s$ determine $F_s$. Using \eqref{+1-1} it is straightforward to see that the value of $u_s^s$ determines all the other entries of $F_s$ and that, since $u_s^s=0$,  we have $F_s=\mathbf0$. Therefore, there is no operator $D$ of odd order in $\mathcal D(W)$ such that $\widetilde D=D$. 

Even more, since $F_s$ is zero \eqref{pol} turns out to be 
 \begin{equation*}
  \begin{aligned}
  \sum_{i=0}^r 
{(-1)^i 2^{2i}\left(\tfrac{n+1}{2}+{\widetilde m}+k\right)_i [k]_i}
(p+{\widetilde m}+2i)
\sum_{\ell=2i}^{s-1} [{\widetilde m}]_{\ell-2i}\; u_{\ell-2i}^\ell&
\\
+
  \sum_{i=1}^{r} 
{(-1)^i 2^{2i}\left(\tfrac{n+1}{2}+{\widetilde m}+k\right)_i [k]_i}
\sum_{\ell=2i-1}^{s-1} [{\widetilde m}]_{\ell-2i+1}\; z_{\ell-2i+1}^\ell
&
-
(p+{\widetilde m})
\sum_{\ell=0}^{s-1} [{\widetilde m}+2k]_\ell\; u_\ell^\ell 
=0.
 \end{aligned}
 \end{equation*}  
If we denote  $s'=s-1$ and observe the term of highest total degree in the variables $\widetilde m$ and $k$, we obtain the following equation,
 \begin{equation*}
    \sum_{i=1}^{r} 
(-1)^i 2^{2i}({\widetilde m}+k)^i k^i
{\widetilde m}^{{s'}-2i+1}\;
\left( u_{{s'}-2i}^{s'} + z_{{s'}-2i+1}^{s'}\right)
-\widetilde m(\widetilde m+2k)^{s'}u_\ell^\ell
=0.
\end{equation*}
If $u_{s'}^{s'}=0$ then this equation is exactly  equation  \eqref{rep} with $s'$ in place of $s$; therefore with the same procedure as before we obtain an equation similar to \eqref{uz}, with $s'$ in place of $s$; if $v_{s'}^{s'}=0$, then we have also an equation similar to \eqref{vy}. Equation \eqref{k} and the 
subsequent equations \eqref{zy} and \eqref{+1-1} are obtained independently of the parity of $s$, therefore they remain valid for $s'$, having then that if $u_{s'}^{s'}=v_{s'}^{s'}=0$ then $F_{s'}={\bf 0}$.
 
Finally, for a given operator $D=\sum_{\ell=0}^s \partial^\ell F_\ell(x)\in \mathcal D(W)$, with $\widetilde D=D$, of even order $s=2r$ we can consider the operator 
$$
D-u_s^sD_1^r-v_s^sD_2^r,
$$
which has to be of order $s-2$. Inductively it can be easily seen that the set of operators 
$$
\left\{
D_1,D_2
\right\}
$$
generates the subalgebra of all the operators in $\mathcal D(W)$ such that $\widetilde D=D$.

The following theorem summarizes all the results obtained up to this point in this section.

\begin{thm}\label{T}
 Any operator $D=\sum_{\ell=0}^s\partial^\ell\,F_\ell\in\mathcal D(W)$ of order $s$, such that $\widetilde D=D$, is of even order and the polynomial $F_s$ is of degree $s$.
 
 Furthermore, the operators $D_1$ and $D_2$ generate the  subalgebra 
$
\left\{ 
D\in\mathcal D(W): \widetilde D=D
\right\}
$
of $\mathcal D(W).$
\end{thm}

\begin{cor}\label{C}
Any operator $D=\sum_{\ell=0}^s\partial^\ell\,F_\ell\in\mathcal D(W)$ of order $s$, such that $\widetilde D=-D$, is of even order and the polynomial $F_s$ is of degree $s$.
\end{cor}
\begin{proof}
Given an operator $D=\sum_{\ell=0}^s\partial^\ell F_\ell\in\mathcal D(W)$ of order $s$ such that $\widetilde D=-D$ let us consider the operator 
$$D_0=D(D_3+D_4).$$
Since $\widetilde D_3 =-D_3$ and $\widetilde D_4 =-D_4$ we have that $\widetilde D_0 =D_0$. Also, the highest order term of $D_3+D_4$ is 
$\partial^2(x^2-1)
\left(\begin{smallmatrix}
                   0&1\\1&0
                  \end{smallmatrix}\right)
$, therefore the order of $ D_0$ is $s+2$, which implies, by Theorem \ref{T}, that $s+2$ is even and that if $D_0=\sum_{\ell=0}^{s+2}\partial^\ell G_\ell$ then $G_{s+2}=(1-x^2)F_s$ is a polynomial of degree $s+2$, hence $F_s$ is a polynomial of order $s$.
\end{proof}

Now we give some results that are pretty technical. In the proofs we strongly use the relations in Remark \ref{rem}.
In particular we have
\begin{align*}
(D_1+D_2)^kD_1=&D_1^{k+1},& (D_1+D_2)^kD_2=&D_2^{k+1},&
(D_1+D_2)^kD_3&=D_2^{k}D_3,&(D_1+D_2)^kD_4&=D_1^{k}D_4,
\end{align*}
thus, looking at \eqref{eigen}, it is also clear that 
 $$ \left\{ (D_1+D_2)^iD_1,(D_1+D_2)^iD_2,(D_1+D_2)^iD_3,(D_1+D_2)^iD_4,\mathbf I: i\in\mathbb N_0 \right\}$$
 is a linear independent set over $\mathbb C$.
\begin{thm}\label{TT}
 The algebra $\mathcal D(W)$ is generated by 
 $
\left\{
D_1,D_2,D_3,D_4
\right\}
$.
If $D=\sum_{\ell=0}^s\partial^\ell\,F_\ell\in\mathcal D(W)$ with $F_s\ne \mathbf0$ then $s$ is even and $F_s$ is of degree $s$.
\end{thm}
\begin{proof}
For any $D\in\mathcal D(W)$ we have that $D=E_1+E_2$, with $\widetilde E_1=E_1$ and $\widetilde E_2=-E_2$.
Thus, it only remains to prove that if $D=\sum_{\ell=0}^s\partial^\ell\,F_\ell\in\mathcal D(W)$, with $\widetilde D=-D$, then it is generated by $ \left\{ D_1,D_2,D_3,D_4 \right\}$. 
 Let $F_s=\sum_{j=0}^s F_j^s$. From Corollary \ref{C} we know that $s$ is even and from the relation $\widetilde D=-D$ it follows that $(F_s^s)_{11}=(F_s^s)_{22}=0$.
 If $s\le2$ the assertion is already proved by Lemma \ref{PZ}, if $s\ge4$ we consider the operator 
 $$
 D_0=D-(D_1+D_2)^{\frac s2-1}
 \left(
 \left(F_s^s\right)_{21}D_3+
 \left(F_s^s\right)_{12}D_4
\right) .
$$
It can be easily checked that $\widetilde D_0=-D_0$ and that $D_0=\sum_{\ell=0}^s\partial^\ell\,G_\ell$ with $G_s$ is of degree less than $s$ and therefore, by Corollary \ref{C}, $G_s=0$, then $D_0$ is an operator of order at most $s-2$, hence by induction we obtain the statement in the theorem.
\end{proof}
With a very similar proof we obtain  the following corollary.
\begin{cor}\label{vector}
 The vector space $\mathcal D(W)$ is spanned  by the set
 $$ \left\{ (D_1+D_2)^iD_1,(D_1+D_2)^iD_2,(D_1+D_2)^iD_3,(D_1+D_2)^iD_4,\mathbf I: i\in\mathbb N_0 \right\}.$$
\end{cor}
\begin{proof}
 Let $D\in\mathcal D(W)$, then it is of the form
 $$D=\sum_{\ell=0}^{2r}\partial^\ell\,F_\ell,$$
 for some $r\in\NN_0$, with $F_s=\sum_{j=0}^s F_j^s$. 
 Let us consider the operator 
 $$
 D_0=D-(D_1+D_2)^{k-1}
 \left(
 \left(F_s^s\right)_{11}D_1+
 \left(F_s^s\right)_{22}D_2+
 \left(F_s^s\right)_{21}D_3+
 \left(F_s^s\right)_{12}D_4
\right) .
$$
It can be easily checked that if $D_0=\sum_{\ell=0}^s\partial^\ell\,G_\ell$ then $G_s$ is of degree smaller than $s$, and therefore by Theorem \ref{TT} $D_0$ is an operator of order at most $s-2$. Again, by induction we obtain the statement in the theorem.
\end{proof}

\begin{cor}\label{basis}
$\mathcal D(W)$ is a free left-module over the subalgebra generated by $D_1+D_2$. The following set is a basis 
 $$ \left\{ \mathbf I,D_2,D_3,D_4\right\}.$$
\end{cor}

As a direct consequence of Proposition 2.2 and Theorem 4.3 we have
\begin{cor}\label{cor}
For any $D\in D(W)$, the order of $D$ equals the degree of $\Lambda_w(D)$ as a polynomial in $w$.
\end{cor}

\section{Structure of $\mathcal D(W)$}\label{Str}

In this section we go deeper into the structure of the algebra $\mathcal D(W)$. Also, we prove that the non-similar weights $W_{p,n}$ have isomorphic algebras $\mathcal D(W_{p,n})$.

\begin{thm}
 The algebra $\mathcal D(W)$ is generated by 
$$ \left\{ D_1+D_2,D_3+D_4 \right\}.$$ 
\end{thm}
\begin{proof}
From Remark \ref{rem} we have that
$$(D_3+D_4)^2=(D_1+D_2)^2+(n-2p)(D_1-D_2),$$
therefore  $ \left\{ D_1+D_2,D_3+D_4 \right\}$ generates $D_1$ and $D_2$.

From
$$(D_3+D_4)D_1=D_3D_1 =D_2D_3-(n-2p)D_3=D_2(D_3+D_4)-(n-2p)D_3$$
we have that $ \left\{ D_1+D_2,D_3+D_4 \right\}$ generates also $D_3$ and  $D_4$. The theorem follows from Theorem \ref{TT}.
\end{proof}
Let us define 
\begin{align*}
A=(D_1+D_2)/(2n-p),&& B=(D_3+D_4)/(2n-p).
\end{align*}

Therefore we have the following expressions for the operators $D_1,D_2,D_3,D_4$ (see Remark \ref{rem})

\begin{equation}\label{ds}
\begin{aligned}
D_1=&({ {{B}^{2}-{A}^{2}+ A}})({n-2p}) /2        ,&      D_2=  & ({ {{A}^{2}-{B}^{2}+  A}} )({n-2p}) /2                  ,\\
D_3=&  { {(-B A+A B+  B}})({n-2p}) /2       ,&      D_4=&  { {(B A-A B+  B)({n-2p}) /2       }}               .
\end{aligned}
\end{equation}

Let us call $\mathcal A$ to the subalgebra generated by $A$. 

\begin{lem}The set $S=\{\mathbf I,B^2, B,BA\}$ is also a basis of $\mathcal D(W)$ as a left-module over $\mathcal A$.
\end{lem}
\begin{proof}
It is clear from \eqref{ds}  and Corollary \ref{basis}  that $S$ generates $\mathcal D(W)$. Let assume that 
$$m_1I+m_2B^2+m_3B+m_4BA=0 $$
for some $m_1,m_2,m_3,m_4\in\mathcal A$. From \eqref{ds} we observe that
\begin{align*}
B^2=A^2  +A-2({n-2p})^{-1}\,D_2 ,   &   &   BA=(D_4-D_3)(n-2p)^{-1}+A(D_3+D_4)  ,\\
\end{align*}
then we have 
\begin{multline*}
(m_1+m_2A^2 I +A)\mathbf I+m_2(-2({n-2p})^{-1})D_2 +\left(m_3+m_4{(-(n-2p)^{-1}\mathbf I+A)}\right)D_3 
+\\
\left(m_3+m_4({(n-2p)^{-1}\mathbf I+A)}\right)D_4 =0,
\end{multline*}
whence it follows in a straightforward way that $m_1=m_2=m_3=m_4=0$.
\end{proof}

It is worth to observe that the following relations hold:
\begin{equation}\label{rel}
\begin{split}
B^2A-AB^2&=0,\\ 
B  {A}^{2}+{A}^{2}  B-2\,A  B  A-B&=0,\\
 B  A  B+{A}^{3}-2\,A  {B}^{2}-A  &=0,\\
 {B}^{3}-2{A}^{2}  B+A  B  A  &=0.
\end{split}
\end{equation}

We shall next prove that the algebra $\mathcal D(W)$ can be described as the complex algebra generated by $A$ and $B$ subject to the relations given in \eqref{rel}.
 Let $V$ be a two dimensional complex vector space and let $T(V)$ be the corresponding tensor algebra. The identity of $T(V)$ will be also denoted by $I$. We choose a basis $\{\alpha ,\beta\}$ of $V$ and define $I(V)$ to be the two sided
ideal of $T(V) $ generated by the  elements on the left-hand side of the following four equations, 
obtained by replacing $A$ by $\alpha$
and $B$ by $\beta$ in \eqref{rel}:
\begin{equation}\label{rell}
\begin{split}
{\beta}^2{\alpha}-{\alpha}{\beta}^2&=0,\\
{\beta}  {{\alpha}}^{2}+{{\alpha}}^{2}  {\beta}-2\,{\alpha}  {\beta}  {\alpha}-{\beta}&=0,\\
 {\beta}  {\alpha}  {\beta}+{{\alpha}}^{3}-2\,{\alpha}  {{\beta}}^{2}-{\alpha}  &=0,\\
 {{\beta}}^{3}-2{{\alpha}}^{2}  {\beta}+{\alpha}  {\beta}  {\alpha}  &=0.
\end{split}
\end{equation}

\begin{thm}
The algebra $\mathcal D(W)$ is isomorphic to the  quotient algebra $T(V)/I(V)$.
\end{thm}

\begin{proof}
For simplicity we will denote with the same symbol an element $T(V)$ and its canonical projection on the quotient algebra $T(V )/I(V )$.
The homomorphism  $\xi$ from 
the quotient algebra $T(V )/I(V )$ 
into $\mathcal D(W)$ defined by $\xi(\alpha) = A$ and $\xi(\beta) = B$  is surjective. Let us denote by $\mathcal A'$ the subalgebra of $T(V)/I(P)$ generated by $\alpha$.  
It is clear that $\mathcal A'$ is isomorphic to $\mathcal A$.
The point now is to prove that $\xi$ is injective. Let $S$  denote the left-module over $\mathcal A$ generated by
$$\{\mathbf I, \beta^2, \beta,\beta\alpha\}.$$ 
Since $\xi$  transforms $S$ into the basis $\{\mathbf I, B^2, B,BA\}$ it is enough to prove that $S=T(V)/I(V)$. First notice that since $\alpha$ and $\beta$ satisfy \eqref{rell} we have that $S$ is invariant under right multiplication by $\alpha$ and left multiplication by $\beta$. If we prove that $S$ is also left invariant under multiplication by $\beta$ we obtain the statement of the theorem. 

We will prove by induction on $k$ that  $\beta\,\alpha^k\,(\alpha\beta)$ belongs to $S$ for any $k\in\NN_0$. For $k=0$ we have $\beta(\beta\alpha)=\beta^2\alpha=\alpha\beta^2$ (first equation in \eqref{rell}). 
For $k=1$ we have $\beta\alpha(\beta\alpha)$, but we can write $\beta\alpha\beta$ as a linear combination of $\beta^2$ and $\mathbf I$ (third equation in \eqref{rell}) which are in the center of the algebra, and then we can write $\beta\alpha\beta\alpha$ also as a  linear combination of  $\beta^2$ and $\mathbf I$.
Assume that $\beta\alpha^k(\beta\alpha)$ can be written as a linear combination of $\mathbf I$ and $\beta^2$ for $k\ge1$ and let us consider the case $k+1$; since $\beta\alpha\alpha$ can be written as $p(\alpha)\beta+q(\alpha)\beta\alpha$ with $p,q$ polynomials (last equation in \eqref{rell}), we have
$$
\beta\alpha^{k+1}\beta\alpha=(p(\alpha)\beta+q(\alpha)\beta\alpha)\alpha^{k-1}\beta\alpha=p(\alpha)\beta\alpha^{k-1}\beta\alpha +q(\alpha)\beta\alpha^{k}\beta\alpha,
$$
which, by the inductive hypothesis, is a linear combination of $\mathbf I$ and $\beta^2$. Then $\beta\,\alpha^k\,(\alpha\beta)$ belongs to $S$ for any $k\in\NN_0$. 

In a very similar way it can be proved that $\beta\,\alpha^k\,(\beta),\beta\,\alpha^k\,(\beta^2)$ and $\beta\,\alpha^k\,(\mathbf I)$ belongs to $S$ for any $k\in\NN_0$. Therefore, $S$ is stable under left multiplication by $\beta$ and the theorem is proved.
\end{proof}

Notice that if one has two similar weights $W$ and $\widetilde W$, then the corresponding differential operator algebras  are isomorphic: Let $M$ be an invertible matrix such that $\widetilde W=M W M^*$, it can be easily checked that $D\to MDM^{-1} $ is an isomorphism from $\mathcal D(W)$ onto $\mathcal D(\widetilde W)$.
The converse is not true. Lemma \ref{neq} states that for $(p,n)\ne(p',n')$ the weights $W_{p,n}$ and $W_{p',n'}$ are non-similar, but relations \eqref{rell} do not depend on the parameters $n$ and $p$. Thus, we have the following result.

\begin{thm}
For any two matrix weights $W=W_{p,n}$ and $W'=W_{p',n'}$  the algebras $\mathcal D(W)$ and $\mathcal D(W')$ are isomorphic.
\end{thm}

\begin{thm}
The algebra $\mathcal D(W)$ is isomorphic to the algebra $\mathcal D(\tilde W)$  studied in \cite{T11},
with $\tilde W=e^{-t^2}e^{tM}e^{tM^*}$ where $$M=\left(\begin{matrix}
0&a\\0&0
\end{matrix}\right), \quad a\in\mathbb C\setminus\{0\},
$$ 
is isomorphic to the algebra $\mathcal D(W)$ considered in the present paper.
\end{thm}
\begin{proof}
In \cite{T11} it is proved that the algebra $\mathcal D(\tilde W)$ is the complex algebra generated by two elements $E$, $F$ subject to certain relations given by equations (109)--(114) (page 321 in \cite{T11}). We observe that equations (110) and (112) in \cite{T11} have minor typos, they actually should be:
\begin{align*}
 {F}^{4}/16 - {a}^{2}{F}^{2}  {E}^{2}/2-2 a{F}^{2}  E-2 {
F}^{2}+{a}^{4}{E}^{4}+8 {a}^{3}{E}^{3}-{a}^{2} \left( {a}^{2}-24
 \right) {E}^{2}-4 a \left( {a}^{2}-8 \right) E-4  \left( {a}^{2}-4
 \right) {\mathbf I}&=0,
\\
{F}^{3}+{F}^{2}  [E,F]-4 {a}^{2}{E}^{2}  F-4 {a}^{2}{E}^{
2}  [E,F]+4 a \left( a-4 \right) E  F+&\\4 a \left( a-4
 \right) E  [E,F]+8  \left( a-2 \right) F+8  \left( a-2
 \right) [E,F]&=0,
\end{align*}
respectively.

Let us consider the algebra morphism $\varphi:\mathcal D(W)\to\mathcal D(\tilde W)$ defined by 
$$\varphi(A)= E +2\mathbf I/|a|^2, \quad \varphi(B)=(F(4+|a|^4)+[E,F](4-|a|^4))/(8|a|^4),$$
and the algebra morphism $\xi:\mathcal D(W)\to\mathcal D(\tilde W)$ defined by
$$\xi(E){=} A-2\mathbf I/|a|^2 , \quad \xi(F) {=} (4+|a|^4)B+(4-|a|^4)[B,A]
,$$

Using \eqref{rel} and (109)--(114) (from \cite{T11}), it is straightforward to see that $\varphi$ and $\xi $ are well defined and that $\varphi$ is an isomorphism with inverse $\xi$.
\end{proof}

\section{The center of $\mathcal D(W)$}\label{Cen}
In this section we determine the center of our algebra and we will analyze the structure of  $\mathcal D(W)$ as a module over its center.

Let us define
\begin{align*}
  C_1=(D_3+D_4)^2, && {C_2}=D_3D_1D_4+D_4\left(D_2-(n-2p)\mathbf I\right)D_3.
\end{align*}
From Lemma \ref{PZ} it follows that 
\begin{equation}\label{c1c2}
\begin{split}
 \Lambda_w(C_1)=
&p_{1}(w)\mathbf I=
 (w+p)(w+p+1)(w+n-p+1)(w+n-p)\mathbf I,
\\
 \Lambda_w({C_2})=
&p_{2}(w)\mathbf I=
\left( w+p \right)  \left( w+p+1 \right) ^{2} \left( w+n-p+1 \right) 
 \left( w+n-p \right) ^{2}
\mathbf I,
\end{split}
\end{equation}
 whence it is clear that ${C_1}$ and ${C_2}$ are in the center of $\mathcal D(W)$. Also, we observe that they are not algebraically independent since the following relations hold:
$$
C_1^3-C_2^2=(n-2p)\,C_1\,C_2.
$$
\begin{thm}
 The center $\mathcal {Z}(W)$ of  $\mathcal D(W)$ is generated by ${C_1}$ and ${C_2}$ and is isomorphic to the affine algebra of the following singular algebraic curve:
$$
x^3-y^2=(n-2p)\,x\,y.
$$
\end{thm}
\begin{proof}
Let us assume that $X$ is an element in $\mathcal {Z}(W)\subset\mathcal D(W)$, let $s$ be the order of $X$. 
Since $\Lambda_w(X)\Lambda_w(D_1) = \Lambda_w(D_1)\Lambda_w(X)$ we have that $\Lambda_w(X)$ is a diagonal matrix for every $w$ (see Theorem \ref{PZ} for the expression of $\Lambda_w(D_1)$). From $\Lambda_w(X)\Lambda_w(D_3)=\Lambda_w(D_3)\Lambda_w(X)$ we obtain that $\Lambda_w(X)$ is a scalar matrix, i.e.
$$\Lambda_w(X)=p_X(w)\mathbf I$$
with $p_X(w)$ a polynomial on $w$. From Corollary \ref{cor} we have that the degree of $p_X$ is equal to the order of $X$, which has to be an even number $s$ (Theorem \ref{TT}). If $s=0$ there is nothing to prove; it is clear that $s\ne2$ since all the operators of order $2$ in $\mathcal D(W)$ are a linear combination of $\{D_1,D_2,D_3,D_4,\mathbf I\}$ (see Theorem \ref{PZ}). We will complete the proof by induction on the even number $s$: assume that $s\ge4$, since the polynomials $p_{1}$ and $p_{2}$ are of degree $4$ and $6$ respectively we can find  integers $s_1\in\NN_0$ and $s_2\in\{0,1\}$ such that the polynomial $p_{\widetilde X}=p_X-p_{1}^{s_1}p_{2}^{s_2}$ is of degree smaller than $s$. Therefore, by the inductive hypothesis, the operator $ \widetilde X=X-{C_1}^{s_1}{C_2}^{s_2}$ is a polynomial in ${C_1}$ and ${C_2}$ and hence also the operator $X$.
 \end{proof}

\begin{cor}\label{poly}
For every operator  $D\in\mathcal Z(W)$ there exists a unique pair of polynomials $p$ and $q$ such that
$$
D=p(C_1)+q(C_1)C_2.
$$
\end{cor}

%

\begin{thm}\label{az}
 The algebra $\mathcal D(W)$ is a finitely generated torsion-free module over the ring $\mathcal Z(W)$. It is not flat and hence is not projective. 
\end{thm}

\begin{proof}
Torsion free is clear since for every operator $D$ in $\mathcal Z(W)$ the eigenvalue $\Lambda_w( D)$ is a scalar matrix.
From Corollary \ref{vector}, since $C_1=(D_3+D_4)^2=(D_1+D_2)^2+(n-2p)(D_1-D_2)$ (see Remark \ref{rem}), it is clear that $\mathcal D(W)$ is finitely generated over the subalgebra generated by $C_1$ and therefore by $\mathcal Z(W)$.
Given that the center $\mathcal Z(W)$ is a Noetherian ring,  $\mathcal D(W)$ is projective (resp. flat) if and only if it is locally free.

We will prove that $\mathcal D(W)$ is not locally free. Let $\mathfrak m$ the maximal ideal generated by $C_1$ and $C_2$, then the complement $S$ are the operators of the form $p(C_1)+q(C_1)C_2$ with $p$ and $q$ polynomials such that the constant term of $p$ is nonzero; i.e. $S$ is the set of  operators $D$ in  $\mathcal Z(W)$ such that the polynomial $\Lambda_w(C_1)$ does not divides $\Lambda_w(D)$ (recall from \eqref{c1c2} that $\Lambda_w(C_1)$ divides $\Lambda_w(C_2)$).

Let us assume that $\{A_i\}_{j\in J}$ is a basis of the localization of $\mathcal D(W)$ with respect to $\mathfrak m$, called $\mathcal D(W)_{\mathfrak m}$, over the ring $\mathcal Z(W)_{\mathfrak m}$, the localization of $\mathcal Z(W)$ by $\mathfrak m$ . Therefore, 
\begin{equation}\label{d2d3}
 D_3=\sum_{j\in J}C_{1,j}A_j\,,\qquad  D_2D_3= \sum_{j\in J}C_{2,j}A_j,
\end{equation}
with $C_{i,j}$ of the form
$$C_{i,j}=\frac{p_{i,j}(C_1)+q_{i,j}(C_1)C_2}{p'_{i,j}(C_1)+q'_{i,j}(C_1)C_2},$$
with $p_{i,j},q_{i,j},p'_{i,j},q'_{i,j}$ polynomials  for all $i=1,2$ and $j\in J$, such that the constant coefficient of $p'_{i,j}$ is nonzero.

Notice that there exists $j_0$ such that $p_{2,j_0}$ has a nonzero constant term, in other words $p_{j_0}\in S$. By considering the eigenvalues in \eqref{d2d3} we have 
\begin{multline*}
\Lambda_w(D_2)\Lambda_w(D_3)\prod_{j\in J} (p'_{2,j}(\Lambda_w(C_1))+q'_{2,j}(\Lambda_w(C_1))\Lambda_w(C_2)) \\
= \sum_{j\in J}
\left(\prod_{k\ne j} (p'_{2,k}(\Lambda_w(C_1))+q'_{2,k}(\Lambda_w(C_1))\Lambda_w(C_2)) \right)
\Big(p_{2,j}(\Lambda_w(C_1))+q_{2,j}(\Lambda_w(C_1))\Lambda_w(C_2))\Big)
\Lambda_w(A_j),
\end{multline*}
if $p_{2,j}$ had a zero constant term for every $j$ then
we would have that the polynomial $w+n-p+1$ divides the right-hand side of the equation above since it divides $\Lambda_w(C_1)$ and $\Lambda_w(C_2)$, hence we would have that it also divides the left-hand side, which is clearly false (see \eqref{eigen}.

Given that 
$$C_1 D_2 D_3-C_2 D_3=0 ,$$
we have 
$$C_1 \sum_{j\in J}C_{2,j}A_j-C_2 \sum_{j\in J}C_{1,j}A_j=0, $$
in particular 
$$
C_1 C_{2,j_0}\hfill-C_2 C_{1,j_0}=0.
$$
Now, we only work out the last identity. Denoting $p_{i,j},q_{i,j},p'_{i,j},q'_{i,j}$ the respective polynomials given by $p_{i,j}(C_1),q_{i,j}(C_1),p'_{i,j}(C_1),q'_{i,j}(C_1)$ (in order to simplify the notation in this part of the proof) and using that $C_2^2=C_1^3+(2n-p)C_1C_2$ we have
\begin{flalign*}
C_1 C_{2,j_0}\hfill-C_2 C_{1,j_0}=&0, \\
C_1(p_{2,j_0}+q_{2,j_0}C_2)/(p'_{2,j_0}+q'_{2,j_0}C_2)=&C_2(p_{1,j_0}+q_{1,j_0}C_2)/( p'_{1,j_0}+ q'_{1,j_0}C_2),&\\
C_1(p_{2,j_0}+q_{2,j_0}C_2)( p'_{1,j_0}+ q'_{1,j_0}C_2)=&(p_{1,j_0}C_2+q_{1,j_0}C_2^2)(p'_{2,j_0}+q'_{2,j_0}C_2),&\\
C_1(p_{2,j_0}+q_{2,j_0}C_2)( p'_{1,j_0}+ q'_{1,j_0}C_2)=&(p_{1,j_0}C_2+q_{1,j_0}(C_1^3-(n-2p)C_1C_2))(p'_{2,j_0}+q'_{2,j_0}C_2),&\\
C_1(p_{2,j_0}+q_{2,j_0}C_2)( p'_{1,j_0}+ q'_{1,j_0}C_2)=&(p_{1,j_0}C_1^3+(q_{1,j_0}-p_{1,j_0}(n-2p)C_1)C_2)(p'_{2,j_0}+q'_{2,j_0}C_2),&\\
p'_{1,j_0}C_1p_{2,j_0}+q'_{1,j_0}C_1q_{2,j_0}C_1^3+C_1( p'_{1,j_0}q_{2,j_0}&+q'_{1,j_0}p_{2,j_0}+(n-2p)q'_{1,j_0}q_{2,j_0})C_2 =&\\
p_{1,j_0}C_1^3p'_{2,j_0}+(q_{1,j_0}-p_{1,j_0}(&n-2p)C_1)q'_{2,j_0}C_1^3 +\\
(p_{1,j_0}C_1^3q'_{2,j_0}+(q_{1,j_0}-p_{1,j_0}(&n-2p)C_1)p'_{2,j_0}+(q_{1,j_0}-p_{1,j_0}(n-2p)C_1)q'_{2,j_0}(n-2p)C_1)C_2,
\end{flalign*}
 but the last equation contradicts Corollary \ref{cor}: let us focus on the term without $C_2$, on the left-hand side of the last equation  we have a polynomial on $C_1$ with nonzero term of degree one but on the right-hand side we have a polynomial with null term of order one. 

This contradiction is a consequence of the assumption of the existence of a basis. Therefore, $\mathcal D(W)_{\mathfrak m}$ is not free. Hence $\mathcal D(W)$ is not locally free over $\mathcal Z(W)$ as asserted.
\end{proof}

\section{On the Weyl algebra}\label{weyl}

In this final section we exhibit a different way to regard
 the algebra $\mathcal D(W)$. We show that under certain circumstances, one has that the algebra $\mathcal D(W)$, for a given weight $W$, is the centralizer of its center in a bigger algebra, the Weyl algebra.
 \smallskip
    
Notice that, for any weight $W$ of size $N$, the algebra $\mathcal D(W)$ is a subalgebra of the Weyl algebra $D$ over $M_N$
of all linear right-hand side ordinary differential operators with coefficients in $M_N[x]$:
$${\mathbf D} =\{D=\sum_{i=0}^s\partial^i F_i: F_i \in M_N[x]\}.$$

\begin{thm}
 Given a weight $W$, if there exists an operator $C\in\mathcal D(W)$ such that $\Lambda_n(C)$ is a scalar matrix for every $n$ and
 $\Lambda_k(C)=\Lambda_j(C)$ only if $k=j$, then 
 $$
C_{\mathbf{D}}\left( C\right)=\mathcal D(W).
$$
In particular, $
C_{\mathbf{D}}\left( \mathcal Z(W)\right)=\mathcal D(W).
$
\end{thm}
\begin{proof}
  Let $D\in C_{\mathbf{D}}\left( C\right)$, for a fixed $n\in\NN_0$ we have 
$$
P_nD=\sum_{j=0}^{K }A_{n,j} P_j,
$$
with $\{A_{n,j}\}$ matrices and $K\in\NN_0$. Since $DC=CD$, we have 
$$
\sum_{j=0}^K A_{n,j}\,\Lambda_j(C) \,P_j=\sum_{j=0}^K \Lambda_n(C)\,A_{n,j} \,P_j,
$$
thus
$$
A_{n,j}\,\Lambda_j(C)=\Lambda_n(C)\,A_{n,j}, \qquad \text{for every }j.
$$
Therefore, $A_{n,j}=0$ for all $j\neq n$. Having then $P_nD= A_{n,n} P_n$, which implies $D\in\mathcal D(W)$. Therefore $
C_{\mathbf{D}}\left( C\right)\subset\mathcal D(W).
$

Since $\Lambda_n(C)$ is a scalar matrix, it is clear that $
C_{\mathbf{D}}\left( C\right)\supset\mathcal D(W)$. The first assertion follows. For the second assertion it is enough to observe that in general we have $
C_{\mathbf{D}}\left( C\right)\supset C_{\mathbf{D}}\left(\mathcal Z(W)\right)
\supset\mathcal D(W)$. 
\end{proof}
With a similar proof, we can obtain the following result.
\begin{thm}
 Given a weight $W$, if for every $n$ there exists an operator $C_n\in\mathcal Z(W)$ such that $\Lambda_k(C_n)$ is a scalar matrix for every $k$ and
 $\Lambda_k(C_n)=\Lambda_n(C_n)$ only if $k=n$, then 
 $$
C_{\mathbf{D}}\left( \mathcal Z(W)\right)=\mathcal D(W).
$$
\end{thm}

\begin{cor}
 For the weight considered in the present paper and the weight considered in \cite{T11} we have 
  $$
C_{\mathbf{D}}\left( \mathcal Z(W)\right)=\mathcal D(W).
$$
\end{cor}

 \bibliographystyle{alpha}

\bibliography{ref15}

\begin{thebibliography}{GdlIMF11}

\bibitem[AKdlR15]{AKR15}
N.~Aldenhoven, E.~Koelink, and A.~M. de~los R{\'{\i}}os.
\newblock Matrix-valued little {$q$}-{J}acobi polynomials.
\newblock {\em J. Approx. Theory}, 193:164--183, 2015.

\bibitem[AS65]{AS65}
M.~Abramowitz and I.~Stegun.
\newblock {\em Handbook of Mathematical Functions with Formulas, Graphs, and
  Mathematical Tables}.
\newblock Dover, New York, 1965.

\bibitem[BC23]{BC23}
J.~L. {Burchnall} and T.~W. {Chaundy}.
\newblock {Commutative ordinary differential operators.}
\newblock {\em {Proc. Lond. Math. Soc. (2)}}, 21:420--440, 1923.

\bibitem[BC28]{BC28}
J.~L. {Burchnall} and T.~W. {Chaundy}.
\newblock {Commutative ordinary differential operators.}
\newblock {\em {Proc. R. Soc. Lond., Ser. A}}, 118:557--583, 1928.

\bibitem[{Boc}29]{B29}
S.~{Bochner}.
\newblock {\"Uber Sturm-Liouvillesche Polynomsysteme.}
\newblock {\em {Math. Z.}}, 29:730--736, 1929.

\bibitem[CdlI14]{CI14}
M.~Cafasso and M.~D. de~la Iglesia.
\newblock Non-commutative {P}ainlev\'e equations and {H}ermite-type matrix
  orthogonal polynomials.
\newblock {\em Comm. Math. Phys.}, 326(2):559--583, 2014.

\bibitem[CG05]{CG05}
M.~M. Castro, , and F.~A. Gr{\"u}nbaum.
\newblock {Orthogonal matrix polynomials satisfying first order differential
  equations: a collection of instructive examples}.
\newblock {\em J. Nonlinear Math. Physics}, 12(2):63--67, 2005.

\bibitem[CG06]{CG06}
M.~M. Castro, , and F.~A. Gr{\"u}nbaum.
\newblock {The algebra of differential operators associated to a given family
  of matrix valued orthogonal polynomials: five instructive examples}.
\newblock {\em Int. Math. Res. Not.}, 27(2):1--33, 2006.

\bibitem[DdlI08a]{DI08}
A.~J. Dur{\'a}n and M.~D. de~la Iglesia.
\newblock Second-order differential operators having several families of
  orthogonal matrix polynomials as eigenfunctions.
\newblock {\em Int. Math. Res. Not. IMRN}, pages Art. ID rnn 084, 24, 2008.

\bibitem[DdlI08b]{DdI08}
A.~J. Dur\'an and M.D. de~la Iglesia.
\newblock Some examples of orthogonal matrix polynomials satisfying odd order
  differential equations.
\newblock {\em Journal of Approximation Theory}, 150(2):153--174, 2008.

\bibitem[DG86]{DG86}
J.~J. Duistermaat and F.~A. Grünbaum.
\newblock Differential equations in the spectral parameter.
\newblock {\em Comm. Math. Phys.}, 103(2):177--240, 1986.

\bibitem[DG04]{DG04}
A.~J. Dur\'an and F.~A. Gr{\"u}nbaum.
\newblock Orthogonal matrix polynomials satisfying second-order differential
  equations.
\newblock {\em Int. Math. Res. Not.}, (10):461--484, 2004.

\bibitem[DG05a]{DG05a}
A.~J. Dur\'an and F.~A. Gr{\"u}nbaum.
\newblock {A characterization for a class of weight matrices with orthogonal
  matrix polynomials satisfying second-order differential equations.}
\newblock {\em {Int. Math. Res. Not.}}, 23:1371--1390, 2005.

\bibitem[DG05b]{DG05c}
A.~J. Dur{\'a}n and F.~A. Gr{\"u}nbaum.
\newblock Orthogonal matrix polynomials, scalar-type {R}odrigues' formulas and
  {P}earson equations.
\newblock {\em J. Approx. Theory}, 134(2):267--280, 2005.

\bibitem[DG05c]{DG05b}
A.~J. Dur\'an and F.~A. Gr{\"u}nbaum.
\newblock {Structural formulas for orthogonal matrix polynomials satisfying
  second-order differential equations. I.}
\newblock {\em {Constr. Approx.}}, 22(2):255--271, 2005.

\bibitem[DG05d]{DG05d}
A.~J. Dur{\'a}n and F.~A. Gr{\"u}nbaum.
\newblock A survey on orthogonal matrix polynomials satisfying second order
  differential equations.
\newblock {\em J. Comput. Appl. Math.}, 178(1-2):169--190, 2005.

\bibitem[dlI11]{I11}
M.~D. de~la Iglesia.
\newblock Some examples of matrix-valued orthogonal functions having a
  differential and an integral operator as eigenfunctions.
\newblock {\em J. Approx. Theory}, 163(5):663--687, 2011.

\bibitem[Dur97]{D97}
A.~J. Dur\'an.
\newblock {Matrix inner product having a matrix symmetric second-order
  differential operator}.
\newblock {\em Rocky Mt. J. Math.}, 27(2):585--600, 1997.

\bibitem[GdlIMF11]{GIV11}
F.~A. Gr{\"u}nbaum, M.~D. de~la Iglesia, and Andrei
  Mart{\'{\i}}nez-Finkelshtein.
\newblock Properties of matrix orthogonal polynomials via their
  {R}iemann-{H}ilbert characterization.
\newblock {\em SIGMA Symmetry Integrability Geom. Methods Appl.}, 7:Paper 098,
  31, 2011.

\bibitem[GH97]{GH97}
F.Alberto {Gr\"unbaum} and Luc {Haine}.
\newblock {A theorem of Bochner, revisited.}
\newblock In {\em {Algebraic aspects of integrable systems: in memory of Irene
  Dorfman}}, pages 143--172. Boston, MA: Birkh\"auser, 1997.

\bibitem[GPT01]{GPT01}
F.~A. Gr{\"u}nbaum, I.~Pacharoni, and J.~Tirao.
\newblock A matrix-valued solution to {B}ochner's problem.
\newblock {\em J. Phys. A}, 34(48):10647--10656, 2001.

\bibitem[GPT02a]{GPT02a}
F.~A. Gr{\"u}nbaum, I.~Pacharoni, and J.~Tirao.
\newblock Matrix valued spherical functions associated to the complex
  projective plane.
\newblock {\em J. Funct. Anal.}, 188(2):350--441, 2002.

\bibitem[GPT02b]{GPT02b}
F.~A. Gr{\"u}nbaum, I.~Pacharoni, and J.~Tirao.
\newblock Matrix valued spherical functions associated to the three dimensional
  hyperbolic space.
\newblock {\em Internat. J. Math.}, 13(7):727--784, 2002.

\bibitem[GPT03]{GPT03}
F.~A. Gr{\"u}nbaum, I.~Pacharoni, and J.~Tirao.
\newblock Matrix valued orthogonal polynomials of the {J}acobi type.
\newblock {\em Indag. Math. (N.S.)}, 14(3-4):353--366, 2003.

\bibitem[GPT04]{GPT04}
F.~A. Gr{\"u}nbaum, I.~Pacharoni, and J.~Tirao.
\newblock An invitation to matrix-valued spherical functions: Linearization of
  products in the case of complex projective space $p_2(\mathbb c)$.
\newblock Cambridge: Cambridge University Press, 2004.

\bibitem[GPT05]{GPT05}
F.~A. Gr{\"u}nbaum, I.~Pacharoni, and J.~Tirao.
\newblock Matrix valued orthogonal polynomials of {J}acobi type: the role of
  group representation theory.
\newblock {\em Ann. Inst. Fourier (Grenoble)}, 55(6):2051--2068, 2005.

\bibitem[GPZ15]{GPZ15}
F.~A. Gr{\"u}nbaum, I.~Pacharoni, and I.~Zurri{\'a}n.
\newblock Time and band limiting for matrix valued functions, an example.
\newblock {\em SIGMA Symmetry Integrability Geom. Methods Appl.}, 11:Paper 044,
  14, 2015.

\bibitem[Gr{\"u}03]{G03}
F.~A. Gr{\"u}nbaum.
\newblock Matrix valued {J}acobi polynomials.
\newblock {\em Bull. Sci. Math.}, 127(3):207--214, 2003.

\bibitem[GT07]{GT07}
F.~A. Gr{\"u}nbaum and J.~Tirao.
\newblock The algebra of differential operators associated to a weight matrix.
\newblock {\em Integral Equations Operator Theory}, 58(4):449--475, 2007.

\bibitem[Kre49]{K49}
M.~G. Krein.
\newblock Infinite j-matrices and a matrix moment problem.
\newblock {\em Dokl. Akad. Nauk SSSR}, 69(2):125--128, 1949.

\bibitem[Kre71]{K71}
M.~G. Krein.
\newblock Fundamental aspects of the representation theory of hermitian
  operators with deficiency index $(m,m)$.
\newblock {\em AMS Translations, series 2}, 97:75--143, 1971.

\bibitem[{Kri}81]{K81}
I.M. {Krichever}.
\newblock {Algebraic curves and non-linear difference equations.}
\newblock {Integrable systems, selected papers, Lond. Math. Soc. Lect. Note
  Ser. 60, 170-171}, 1981.

\bibitem[KV01]{KV01}
M.~Kimura and P.~Vanhaecke.
\newblock Commuting matrix differential operators and loop algebras.
\newblock {\em Bulletin des Sciences Mathématiques}, 125(5):407 -- 428, 2001.

\bibitem[KvPR12]{KvPR12}
E.~Koelink, M.~van Pruijssen, and P.~Rom{\'a}n.
\newblock Matrix-valued orthogonal polynomials related to {$({\rm
  SU}(2)\times{\rm SU}(2),{\rm diag})$}.
\newblock {\em Int. Math. Res. Not. IMRN}, (24):5673--5730, 2012.

\bibitem[KvPR13]{KvPR13}
E.~Koelink, M.~van Pruijssen, and P.~Rom{\'a}n.
\newblock Matrix-valued orthogonal polynomials related to{$({\rm
  SU}(2)\times{\rm SU}(2),{\rm diag})$} {II}.
\newblock {\em Publ. Res. Inst. Math. Sci.}, 49(2):271--312, 2013.

\bibitem[Mir05]{M05}
L.~Miranian.
\newblock Matrix-valued orthogonal polynomials on the real line: some
  extensions of the classical theory.
\newblock {\em J. Phys. A}, 38(25):5731--5749, 2005.

\bibitem[{Mum}77]{M77}
D.~{Mumford}.
\newblock {An algebro-geometric construction of commuting operators and of
  solutions to the Toda lattice equation, Korteweg de Vries equation and
  related non- linear equations.}
\newblock {Proc. int. Symp. on algebraic geometry Kyoto 1977, 115-153}, 1977.

\bibitem[PR08]{PR08}
I.~Pacharoni and P.~Rom\'an.
\newblock {A sequence of matrix valued orthogonal polynomials associated to
  spherical functions}.
\newblock {\em Constr. Approx.}, 28(2):127--147, 2008.

\bibitem[PT07]{PT07a}
I.~Pacharoni and J.~Tirao.
\newblock {Three term recursion relation for spherical functions associated to
  the complex hyperbolic plane}.
\newblock {\em J. Lie Theory}, 17(4):791--828, 2007.

\bibitem[PZ15]{PZ15}
I.~Pacharoni and I.~Zurri{\'a}n.
\newblock Matrix {G}egenbauer polynomials: The $2\times 2$ fundamental cases.
\newblock {\em Constructive Approximation}, pages 1--19, 2015.

\bibitem[Tak05]{T05}
K.~Takasaki.
\newblock Tyurin parameters of commuting pairs and infinite dimensional
  grassmann manifold.
\newblock {\em arXiv:nlin/0505005}, 2005.

\bibitem[Tir77]{T77}
J.~Tirao.
\newblock {Spherical functions}.
\newblock {\em Rev. Un. Mat. Argentina}, 28:75--98, 1977.

\bibitem[Tir03]{T03}
J.~Tirao.
\newblock The matrix-valued hypergeometric equation.
\newblock {\em Proc. Natl. Acad. Sci. U.S.A.}, 100(14):8138--8141, 2003.

\bibitem[Tir11]{T11}
J.~Tirao.
\newblock {The algebra of differential operators associated to a weight matrix:
  a first example}.
\newblock {Polcino Milies, C\'esar (ed.), Groups, algebras and applications.
  XVIII Latin American algebra colloquium, S\~ao Pedro, Brazil, August 3--8,
  2009. Proceedings. Providence, RI: American Mathematical Society (AMS).
  Contemporary Mathematics 537, 291-324}, 2011.

\bibitem[TZ14]{TZ14b}
J.~A. {Tirao} and I.~{Zurri\'an}.
\newblock {Spherical functions of fundamental $K$-types associated with the
  $n$-dimensional sphere.}
\newblock {\em {SIGMA, Symmetry Integrability Geom. Methods Appl.}}, 10:paper
  071, 41, 2014.

\bibitem[TZ15]{TZ15}
J.~Tirao and I.~Zurri\'an.
\newblock {Reducibility of Matrix Weights}.
\newblock {\em ArXiv e-prints}, January 2015.

\bibitem[vPR14]{vPR14}
M.~van {Pruijssen} and P.~Rom{\'a}n.
\newblock Matrix valued classical pairs related to compact {G}elfand pairs of
  rank one.
\newblock {\em SIGMA Symmetry Integrability Geom. Methods Appl.}, 10:Paper 113,
  28, 2014.

\end{thebibliography}

\end{document}